\documentclass[11pt]{amsart}
\usepackage{amsmath}
\usepackage{amssymb}
\usepackage{amsthm}
\usepackage{latexsym,tikz}
\usepackage{hyperref}
\usepackage{enumerate}
\usepackage[all]{xy}

\newtheorem{thm}{Theorem}[section]

\newtheorem{lem}[thm]{Lemma}
\newtheorem{prop}[thm]{Proposition}

\newtheorem{ex}[thm]{Example}

\def\Irr{\mathbf{Irr}}
\def\cG{\mathcal{G}}

\def\I{\mathfrak{I}}

\def\cS{\mathcal{S}}

\def\rZ{\mathrm{Z}}

\def\dim{\mathrm{dim}}
\def\Gal{\mathrm{Gal}}
\def\Hom{\mathrm{Hom}}
\def\im{\mathrm{im}}

\def\Ind{\mathrm{Ind}}
\def\ind{\mathrm{ind}}

\def\span{\mathrm{span}}

\def\sep{\mathrm{s}}
\def\ur{\mathrm{ur}}
\def\ab{\mathrm{ab}}
\def\val{\mathrm{val}}

\def\SO{\mathrm{SO}}

\def\PGL{\mathrm{PGL}}

\def\GL{\mathrm{GL}}

\def\Nrd{\mathrm{Nrd}}
\def\F{\mathbb{F}}

\def\bW{\mathbf{W}}
\def\R{\mathbb{R}}
\def\C{\mathbb{C}}

\def\Z{\mathbb{Z}}

\def\Q{\mathbb{Q}}
\def\SL{\mathrm{SL}}
\def\PSL{\mathrm{PSL}}
\def\im{\mathrm{im}}

\def\res{\mathrm{res}}
\def\Tr{\mathrm{Tr}}

\def\W{\mathbf{W}}
\def\I{\mathbf{I}}
\def\i{\mathbf{i}}
\def\j{\mathbf{j}}
\def\k{\mathbf{k}}

\def\gphi{{\varphi}}
\newcommand{\dv}{d_{\varpi}}
\newcommand{\mf}{\mathfrak}
\newcommand{\mc}{\mathcal}

\newcommand{\matje}[4]{\left(\begin{smallmatrix} #1 & #2 \\ 
#3 & #4 \end{smallmatrix}\right)}

\begin{document}

\title[$L$-packets and depth]{On $L$-packets and depth for $\SL_2(K)$ and its inner form}

\author[A.-M. Aubert]{Anne-Marie Aubert}
\address{Institut de Math\'ematiques de Jussieu -- Paris Rive Gauche, 
U.M.R. 7586 du C.N.R.S., U.P.M.C., 4 place Jussieu 75005 Paris, France}
\email{anne-marie.aubert@imj-prg.fr}
\author[S. Mendes]{Sergio Mendes}
\address{ISCTE - Lisbon University Institute\\    Av. das For\c{c}as Armadas\\     1649-026, Lisbon\\   Portugal}
\email{sergio.mendes@iscte.pt}
\author[R. Plymen]{Roger Plymen}
\address{School of Mathematics, Southampton University, Southampton SO17 1BJ,  England
\emph{and} School of Mathematics, Manchester University, Manchester M13 9PL, England}
\email{r.j.plymen@soton.ac.uk \quad plymen@manchester.ac.uk}
\author[M. Solleveld]{Maarten Solleveld}
\address{IMAPP, Radboud Universiteit, Heyendaalseweg 135, 6525AJ, Nijmegen, the Netherlands}
\email{m.solleveld@science.ru.nl}

\thanks{Maarten Solleveld is supported by a NWO Vidi grant "A Hecke
algebra approach to the local Langlands correspondence" (nr.
639.032.528).}
\date{\today}
\maketitle

\begin{abstract}  
We consider the group $\SL_2(K)$, where $K$ is a local non-archimedean field of characteristic two.  
We prove that the depth of any irreducible representation of $\SL_2 (K)$ is larger than the 
depth of the corresponding Langlands parameter, with equality if and only if the L-parameter
is essentially tame. 

We also work out a classification of all L-packets for $\SL_2 (K)$ and for its non-split
inner form, and we provide explicit formulae for the depths of their L-parameters.
\end{abstract}

\tableofcontents

\section{Introduction}
\smallskip

\smallskip

Let $K$ be a non-archimedean local field and let $K_\sep$ be a separable closure of $K$.  A central role in the 
representation theory of reductive $K$-groups is played by the local Langlands correspondence (LLC).
It is known to exist in particular for the inner forms of the groups $\GL_n(K)$ or $\SL_n(K)$, and 
to preserve interesting arithmetic information, like local L-functions and $\epsilon$-factors. 

Another invariant that makes sense on both sides of the LLC is \emph{depth}. 
The \emph{depth} $d(\pi)$ of an irreducible smooth representation $\pi$ of a reductive $p$-adic group $\cG$ was defined 
by Moy and Prasad \cite{MoPr}  in terms of filtrations $\cG_{x,r}$ ($r \in \R_{\geq 0}$) of
its parahoric subgroups $\cG_x$.
The depth of a Langlands parameter $\phi$  is defined to be the smallest 
number $d(\phi) \geq 0$ such that $\phi$ is trivial on $\Gal (F_s/F)^r$ for all $r > d(\phi)$, where
$\Gal (K_\sep / K)^r$ be the $r$-th ramification subgroup of the absolute Galois group of $K$.
   
Let $D$ be a division algebra with centre $K$, of dimension $d^2$ over $K$.
Then $\GL_m (D)$ is an inner form of $\GL_n(K)$ with $n = dm$. There is a 
reduced norm map Nrd$: \GL_m (D) \to K^\times$ and the derived group 
$\SL_m(D) := \ker (\text{Nrd} \colon\cG \to K^\times)$ is an inner form of $\SL_n(K)$.
Every inner form of $\GL_n(K)$ or $\SL_n(K)$ is isomorphic to one of this kind. When $n=2$, the only possibilities 
for $d$ are $1$ or $2$, and so the inner forms are, up to isomorphism, $\GL_2(K)$ and $D^\times$, and $\SL_2(K)$ and $\SL_1(D)$. 

The LLC for $\GL_m(D)$ preserves the depth, that is, for every smooth irreducible representation $\pi$ of $\GL_m(D)$,
we have $d(\pi)= d(\varphi_\pi)$, where $\varphi_\pi$ corresponds to $\pi$ by the LLC \cite[Theorem~2.9]{ABPS1}.

The situation is different for $\SL_m(D)$. All the irreducible representations in a given L-packet $\Pi_\phi$ have the same depth, so the depth is an invariant of the L-packet, say $d(\Pi_\phi)$. We have $d(\Pi_\phi)=d(\varphi)$ where $\varphi$ is a lift of $\phi$ which has minimal depth among the lifts of $\phi$, and the following holds:
\begin{equation}\label{eq:D}
d(\phi) \le d(\Pi_\phi) 
\end{equation}
for any Langlands parameter $\phi$ for $\SL_m(D)$ \cite[Proposition~3.4 and Corollary~3.4]{ABPS1}.
Moreover \eqref{eq:D} is an equality if
$\phi$ is \emph{essentially tame}, that is,  if the image by $\phi$ of the wild inertia subgroup
$\mathbf P_K$ of the Weil group $\bW_K$ of $K$ lies in a maximal torus of $\PGL_n (\C)$.   

We observe that this notion of essentially tameness is consistent with the usual notion for 
Langlands parameters for $\GL_n(K)$. Indeed, any lift
$\varphi\colon \W_K \to\GL_n(\C)$ of $\phi$, is called essentially tame if its restriction to 
$\mathbf P_K$ is a direct sum of characters.
Clearly $\varphi$ is essentially tame if and only if $\varphi
(\mathbf P_K)$ lies in a maximal torus of $\GL_n (\C)$, which in turn is
equivalent to $\phi (\mathbf P_K)$ lying in a maximal torus of $\PGL_n (\C)$.

We denote by $t(\varphi)$ the \emph{torsion number} of $\varphi$, 
that is, the number of unramified characters $\chi$ of $\W_K$ such $\varphi\chi\cong\varphi$.
Then $\phi$ and $\varphi$ are essentially tame if and only if the residual characteristic $p$ of $K$ does not divide $n/t(\varphi)$ 
\cite[Appendix]{BHet}.

\smallskip 

In this article we take $K$ to be a local non-archimedean field $K$ of characteristic $2$.  In positive characteristic,  $K$ is of the form $K= \F_q((t))$,  the field of 
Laurent series with coefficients in $\F_q$, with $q=2^f$. This case is particularly
interesting because there are countably many quadratic extensions of $\F_q((t))$.   These quadratic extensions are parametrised by the cosets in $K/ \wp(K)$ where $\wp$ is the  map, familiar from Artin-Schreier theory, given by $\wp(X) = X^2 - X$.

We first show that equality holds in (\ref{eq:D}) only if $\phi$ is essentially tame ({\it i.e.,} $t(\varphi) = 2$):

\begin{thm} Let $K$ be a non-archimedean local field of characteristic $2$, and let $\pi$ be an irreducible representation of an inner form 
of $\SL_2 (K)$, with Langlands parameter $\phi$.  If $\phi$ is not essentially tame then we have
  \[
 d(\pi) > d(\phi).
 \]
\end{thm}

Let $\varphi$ be a lift of $\phi$ with minimal depth among the lifts of $\phi$. 
In the proof we distinguish the cases where $\varphi$ is imprimitive, respectively primitive.

An irreducible Langlands parameter $\varphi\colon \W_K\to\GL_2(\C)$ is called \emph{imprimitive} if there exists a separable quadratic extension $L$ of $K$
and a character $\xi$ of $L^\times$ such that $\varphi\simeq\ind_{\W_L}^{\W_K}(\xi)$. Then the depth of $\varphi$ and $\phi$ may be expressed in terms of that of $\xi$ and $\xi^2$, respectively, as
\[
d(\varphi) = \big( d(\xi) + d(L/K) \big) / 2\quad\text{and}\quad d(\phi) = \big( d(\xi^2) + d(L/K) \big) / 2,  
\]
where $\mathfrak p_K^{1 + d(L/K)}$ is the relative discriminant of $L/K$.
Let $\mathfrak{T}(\varphi)$ be the group of characters $\chi$ of $\W_K$ such that $\chi \otimes \varphi \simeq \varphi$. 
As in \cite[41.4]{BH}, we call $\varphi$ \emph{totally ramified} if $\mathfrak{T}(\phi)$ does not contain any unramified character. 
If $\varphi$ is not essentially tame, then it is totally ramified. We check that if this case we have $d(\xi)>d(\xi^2)$, and hence 
$d(\Pi_\phi) > d(\phi)$. 

\smallskip

We obtain in Proposition \ref{prop:3.2} the following characterization of $L$-packets for $\SL_2 (K)$ or 
$\SL_1 (D)$: an $L$-packet is a minimal set of irreducible representations from which a stable 
distribution can be constructed.

Next we give the explicit classification of the L-packets for both $\SL_2(K)$ and $\SL_1(D)$.

In particular, to each biquadratic extension $L/K$, there is attached a Langlands parameter
$\phi = \phi_{L/K}$, and an $L$-packet $\Pi_{\phi}$ of cardinality $4$.   
The  depth of the parameter $\phi_{L/K}$ depends on the extension $L/K$.   
More precisely, the  numbers $d(\phi)$ depend on the breaks in the upper ramification filtration 
of the Galois group $\Gal(L/K) = \Z/2\Z \times \Z/2\Z$. Let $D$ be a central division algebra of 
dimension $4$ over $K$.   The parameter $\phi$ is relevant for the inner form $\SL_1(D)$, 
which admits singleton $L$-packets.

\begin{thm} Let $L/K$ be a biquadratic extension, let $\phi$ be the Langlands parameter $\phi_{L/K}$.   
If the highest break in the upper ramification of the
Galois group $\Gal(L/K)$ is $t$ then we have $d(\phi) = t$.   For every $\pi \in
\Pi_{\phi}(\SL_2(K)) \cup \Pi_{\phi}(\SL_1(D))$  these integers provide lower bounds:
\[
d(\pi) \geq d(\phi).
\]
Depending on the extension $L/K$, all the odd numbers $1,3,5,7,\ldots$ are achieved as such breaks.
\end{thm}


  This contrasts strikingly with the case of $\SL_2(\Q_p)$ with $p>2$.  
  Here there is a unique biquadratic extension $L/K$, and a unique tamely ramified discrete parameter 
  $\phi : \Gal(L/K) \to \SO_3(\R)$ of depth zero.

  Let $E/K$ be the quadratic extension given by
  \[
  E = K(\wp^{-1}(\varpi^{-2n-1}))
  \]
  with $\varpi$ a uniformizer and $n = 0,1,2,3,\ldots$   and let $\phi_E$ be the associated $L$-parameter.   
  We prove in Subsection \ref{par:four} that the depth of $\phi_E$ is given by
  \[
  d(\phi_E) = 2n+1.
  \]
For the $L$-packets considered in this article, the depths $d(\pi)$ can be arbitrarily large.

We have included an Appendix on aspects of the Artin-Schreier theory.   
This Appendix goes a little further than the exposition in 
\cite[p.146--151]{FV} and the article of Dalawat \cite{Da}.   
We have the occasion to refer to the Appendix at several points in our article.

We thank Chandan Dalawat for a valuable exchange of emails and for bringing the reference \cite{Da}
to our attention.

\section{Depth of $L$-parameters}
\label{sec:depth}

The field $K$ possesses a central division algebra $D$ of dimension $4$ and, up to isomorphism, only one.  
The group $D^\times$ is locally profinite and is compact modulo its centre $K^\times$, see \cite[p.325]{BH}. 
Let $\Nrd$ denote the reduced norm on $D^\times$.  Define
\[
\SL_1(D) = \{x \in D^\times : \Nrd (x) = 1\}.
\]
Then $\SL_1(D)$ is an inner form of $\SL_2(K)$.
The articles \cite{HiSa,ABPS2} finalize the local Langlands correspondence for any inner form of 
$\SL_n$ over all local fields.\\

\emph{Depth of an L-parameter for} $\GL_2(K)$.  
Let $\W_K$ denote the Weil group of $K$, and let $\Phi (\GL_2(K))$ be the set of 
L-parameters $\gphi\colon \W_K\times \SL_2 (\C)\to\GL_2 (\C)$ for inner forms of $\GL_2(K)$. 
Let $t$ be a real number, $t\geq 0$, let $\Gal (K_\sep/K)^t$ be the $t$-th ramification subgroup 
of the absolute Galois group of $K$. We define
\begin{equation} \label{eqn:Phit}
\Phi_t(\GL_2(K)) := \{ \gphi \in \Phi (\GL_2(K)) \,:\, \Gal(K_\sep/K)^t \subset \ker (\gphi) \} .
\end{equation}
Notice that $\Phi_{t'}(\GL_2(K)) \subset \Phi_t (\GL_2(K))$, if $t' \leq t$.
It is known that the set of $t$'s at which Gal$(F_s / F)^t$ breaks consists of
rational numbers and is discrete \cite[Chap. IV, \S 3]{Ser}. In particular there exists 
a unique rational number $d(\gphi)$, called the \emph{depth} of $\gphi$, such that
\begin{equation}
\gphi \notin \Phi_{d(\gphi)}(\GL_2(K)) \quad \text{and} \quad
\gphi \in \Phi_{t}(\GL_2(K)) \text{ for any } t > d(\gphi) .
\end{equation}

\emph{Depth of an L-parameter for} $\SL_2(K)$.   
The depth of an L-parameter $\phi\colon\W_K \times \SL_2 (\C)\to\PGL_2 (\C)$ for an inner form of $\SL_2(K)$
is defined as:
\begin{equation} \label{eqn: depthSL}
d(\phi) = \inf \{ t \in \R_{\geq 0} \mid \Gal (K_\sep/K)^{t+} \subset \ker \phi \}, 
\end{equation}
where 
\[
\Gal (K_\sep /K)^{t+}:= \bigcap_{r > t} G^{r}.
\] 
Each projective representation $\phi \colon \W_K \to \PGL_2(\C)$ lifts to a Galois representation
\[
\gphi \colon \W_K \to \GL_2(\C).
\] 
For any such lift $\varphi$ of $\phi$ we have $\ker (\varphi) \subset \ker \phi$, so 
\begin{equation} \label{eqn: dlifts}
d(\gphi) \geq d(\phi) .
\end{equation}

Let $\varphi\colon \W_K\to \GL_2 (\C)$ be a $2$-dimensional irreducible representation of $\W_K$, 
and let $\mathfrak{T}(\varphi)$ be 
the group of characters $\chi$ of $\W_K$ such that $\chi \otimes \varphi \simeq \varphi$. Then $\varphi$ is 
primitive if $\mathfrak{T}(\varphi)=\{1\}$, \emph{simply imprimitive} if $\mathfrak{T}(\varphi)$ 
has order $2$, and
\emph{triply imprimitive} if $\mathfrak{T}(\varphi)$ has order $4$, as in \cite[41.3]{BH}.     
Comparing determinants, we see that every nontrivial element of $\mathfrak{T}(\varphi)$ has order 2.

As in \cite[41.4]{BH}, we call $\phi$ and $\varphi$ \emph{unramified} if $\mathfrak{T}(\varphi) \setminus \{1\}$
contains an unramified character, and \emph{totally ramified} if $\mathfrak{T}(\varphi) \setminus \{1\}$ 
does not contain any unramified character. By definition, a primitive representation is totally ramified.
 Thus every imprimitive irreducible representation of dimension $2$ of $\W_K$ which is not totally 
ramified is essentially tame. 

Let $\phi\colon \W_K\times\SL_2(\C)\to\PGL_2(\C)$ with trivial restriction to $\SL_2(\C)$, and such that 
$\varphi$ is a lift of $\phi$. If $\varphi$ is essentially tame and has minimal depth among the lifts of 
$\phi$, then we have $d(\phi)=d(\varphi)$ \cite[Theorem 3.8]{ABPS1}.
Thus we are reduced to computing the depths of the projective representations of $\W_K$ which lift to 
 totally ramified representations. 

\smallskip

We recall how the depth of an irreducible representation
$(\varphi,V)$ of $\W_K$ can be computed. Put $E = (K_\sep)^{\ker \varphi}$, so that $\phi$ factors 
through Gal$(E/K)$. Let $g_j$ be the order of the ramification subgroup $\Gal(E/K)_j$ 
(in the lower numbering). The Artin conductor $a(\varphi) = a(V)$ is given by
\begin{equation}\label{eq:A1}
a(\varphi) = g_0^{-1} \sum_{j \geq 0} g_j \, \dim \big( V / V^{\Gal (E/K)_j} \big) \in \Z_{\geq 0}.
\end{equation}
Since $(\varphi,V)$ is irreducible and $\Gal(E/K)_j$ is normal in $\Gal (E/K)$, 
$V^{\Gal (E/K)_j} = 0$ whenever $g_j > 1$. Thus \eqref{eq:A1} simplifies to
the formula \cite[(1)]{GrRe}:
\begin{equation}\label{eq:A2}
a(\varphi) = \frac{\dim V}{g_0} \sum_{j \geq 0 : g_j > 1} g_j =
\dim V + \frac{\dim V}{g_0} \sum_{j \geq 1 : g_j > 1} g_j 
\end{equation}
It was shown in \cite[Lemma 4.1]{ABPS2} that
\begin{equation} \label{eqn:depth}
d(\varphi):=\begin{cases}
0 & \text{if } \mathbf{I}_F\subset\ker(\phi), \\
\frac{\displaystyle a (\varphi)}{\displaystyle \dim V}-1 & \text{otherwise.}
\end{cases}
\end{equation}
Let $\varphi \colon \W_K \to \GL_2 (\C)$ be a totally ramified irreducible 
representation. Let $\phi \colon \W_K \to \PGL_2 (\C)$ be its projection. 
We will show that $d(\varphi) > d(\phi)$. To this end we may and will 
assume that $\varphi$ has minimal depth among the lifts of $\phi$.

\begin{thm} Let $\varphi$ be an irreducible totally ramified representation 
$\W_K \to \GL_2(\C)$, let $\phi:  \W_K \to \PGL_2(\C)$ be its projection.   Then we have 
\[
d(\varphi) > d(\phi).
\]
\end{thm}

\begin{proof} \emph{Primitive representations}.  Let $\varphi$ be primitive.
Put $E = K_\sep^{\ker \phi}$ and $E^+ = K_\sep^{\ker \varphi}$. By \cite[\S 42.3]{BH}
there exists a unique intermediate field $K \subset L \subset E$ such that
$E/L$ is a wildly ramified biquadratic extension. Then $\phi (\Gal (E/L))$
is a subgroup of $\PGL_2 (\C)$ isomorphic to the Klein four group. Up to conjugacy 
$\PGL_2 (\C)$ has only one such subgroup. After a suitable change of basis,
we may assume that it is
\begin{equation}\label{eq:D2}
D_2 := \big\{ \matje{1}{0}{0}{1}, \matje{0}{i}{i}{0}, \matje{-i}{0}{0}{i}, 
\matje{0}{1}{-1}{0} \big\} \subset \PGL_2 (\C) .
\end{equation}
The three subextensions of $E/L$ are conjugate under Gal$(E/K)$ because 
the conjugation action of $A_4$ on its normal subgroup $V_4$ of order four
is transitive on the nontrivial elements of $V_4$. Hence there is a unique
$r \in \Z$ such that Gal$(E/L)_r = \Gal(E/L)$ and Gal$(E/L)_{r+1} = \{1\}$.
In section \ref{par:ram} we will see that $r$ is odd.
We call this $r$ the \emph{ramification depth} of $E/L$.

The nontrivial elements of Gal$(E/L)$ are the deepest elements of Gal$(E/K)$
outside the kernel of $\phi$, and therefore the depth of $\phi$ can be expressed
in terms of~$r$.

Let us compare this to what happens for the lift $\varphi$ of $\phi$. Since
$\SL_2 (\C) \to \PGL_2 (\C)$ is a surjection with kernel of order 2, the
preimage of $\phi (\W_K)$ in $\SL_2 (\C)$ has order $2 |\phi (\W_K)|$. The matrices
in \eqref{eq:D2} do not yet form a group in $\GL_2 (\C)$, for that we really need 
the nontrivial element of $\ker (\SL_2 (\C) \to \PGL_2 (\C))$. In other words,
$\SL_2 (\C)$ contains a unique subgroup of order $2 [E:K]$ which projects onto
$\phi (\W_K)$. As $\varphi$ has minimal depth among the lifts of $\phi$,
$\varphi (\W_K)$ is precisely this subgroup. Thus $[E^+:E] = 2$
and $\Gal(E^+/K)$ is a nontrivial index two central extension of $\Gal(E/K)$.
In particular $\Gal (E^+/L)$ is isomorphic to the quaternion group of order eight.

Choose a subset $\{w_1 = 1,w_2,w_3,w_4\} \subset \Gal (E^+ / L)$ which projects 
onto $\Gal (E/L)$. We may assume that the $\varphi (w_i)$ are ordered as in 
\eqref{eq:D2}. As $\ker (\GL_2 (\C) \to \PGL_2 (\C))$ is central,
\[
[\varphi (w_3),\varphi (w_4)] = [\matje{-i}{0}{0}{i}, 
\matje{0}{1}{-1}{0}] = \matje{-1}{0}{0}{-1} \in \GL_2 (\C) .
\]
Write 
\begin{equation} \label{eqn:z}
z = [w_3,w_4] \in \Gal (E^+ / L),
\end{equation} so that $\varphi (z) = \matje{-1}{0}{0}{-1}$.
It follows from the definition of $r$ and the condition on $\varphi$ that 
\[
\Gal (E^+/L)_r = \Gal(E^+/L) \text{ and } \Gal(E^+/L)_{r+1} = \Gal(E^+/E).
\]
By \cite[Proposition IV.2.10]{Ser} $z \in \Gal (E^+/L)_{2r+1}$.
Now $z \notin \ker (\varphi)$ and it lies deeper in Gal$(E^+/K)$ than 
$w_2$, $w_3$ and $w_4$. On the other hand, $z$ does lie in the kernel of $\phi$,
which explains why $\varphi$ has larger depth than $\phi$.

In the sequel of this section, we assume that the depth of the element $z$ defined 
in \eqref{eqn:z} is exactly $2r+1$. This is allowed because, in the above setting, 
it constitutes the worst possible case for the theorem.\\[2mm]
 

\emph{Octahedral representations}.  Let $\varphi$ be octahedral, that is, it is primitive and $\phi (\W_K) \cong S_4$.
Let Ad denote the adjoint representation of $\PGL_2 (\C)$ on 
$\mf{sl}_2 (\C) = \text{Lie}(\PGL_2 (\C))$. Then Ad$\circ 
\phi$ is an irreducible 3-dimensional representation of $\W_K$. Since
$\PGL_2 (\C)$ is the adjoint group of $\mf{sl}_2 (\C)$, Ad$\circ \phi$ has
the same kernel and hence the same depth as $\phi$. 

By \cite[Theorem 42.2]{BH} $L/K$ is Galois with automorphism group $S_3$
and residue degree 2. Thus Ad$(\phi (\I_K)) \subset \mathrm{Ad}
(\phi (\W_K))$ is a normal subgroup of index two, isomorphic to $A_4$.
As $L/K$ has tame ramification index 3, the image of the wild
inertia subgroup $\mathbf P_K$ under Ad$\circ \phi$ equals the image
of Gal$(E/L)$. By our convention \eqref{eq:D2} it is Ad$(D_2)$. By the 
definition of $r$ as the ramification depth of $E/L$, we have 
\[
g_0 = 12 ,\, g_1 = \dots = g_r = 4 \text{ and } g_{r+1} = 1
\]
With the formula \eqref{eq:A2} we find 
\[
a(\text{Ad} \circ \phi) = \frac{3}{12} (12 + r \cdot 4) = 3 + r, 
\]
and from \eqref{eqn:depth} we conclude that 
\[
d(\phi) = d(\text{Ad} \circ \phi) = r/3.
\]
On the other hand, $\varphi$ is an irreducible two-dimensional representation
of $\W_K$, and we must base our calculations on the Galois group of $E^+ / K$.
The numbers
\[
g_j = |\Gal (E^+/K)_j| = |\varphi (\Gal(E^+/K)_j)|
\]
can be computed from those for $\phi$ by means of the twofold covering
$\varphi (\W_K) \to \phi (\W_K)$. We find
\[
g_0 = 24, g_1 = \cdots = g_r = 8 \text{ and } g_{r+1} = \cdots = g_{2r+1} = 2.
\]
Assuming that the depth of $z$ is precisely $2r+1$ (see above), we can also
say that $g_{2r+2} = 1$. Then \eqref{eq:A2} gives
\[
a(\varphi) = \frac{2}{24} (24 + r \cdot 8 + (r+1) \cdot 2 ) = 2 + \frac{5 r + 1}{6} .
\]
Now \eqref{eqn:depth} says that 
\[
d(\varphi) = (5 r + 1) / 12 .
\] 
We note that this is strictly larger than $d(\phi) = r/3$.
(As $a(\phi) \in \Z_{\geq 0}$, we must have $r - 1 \in 6 \Z$. This means that
above not all biquadratic extensions can occur.)\\[2mm]

\emph{Tetrahedral representations}.   
Let $\varphi$ be tetrahedral, that is, it is primitive and $\phi (\W_K) \cong A_4$.
By \cite[Theorem 42.2]{BH} $L / K$ is a cubic Galois extension. It is of prime
order, so either it is unramified or it is totally ramified.

First we consider the case that $L/K$ ramifies totally. Then $\I_K$ surjects
onto Gal$(E/K)$, so $\varphi (\I_K) = \varphi (\W_K)$. This means that
within $\I_K$ everything is similar to octahedral representations. 
The same calculations as above show that
\[
d(\phi) = r/3 < d(\varphi) = (5r+1) / 12 .
\]
Now we look at the case where $L/K$ is unramified. Then 
\[
\phi (\I_K) = \phi (\Gal (E/K)) = D_2.
\]
To compute the depth, we replace $\phi$ by the 3-dimensional representation
Ad$\circ \phi$ of $\W_K$ on $\mathfrak{sl}_2 (\C)$. With $r$ as before, 
$g_0 = \cdots = g_r = 4$ and $g_{r+1} = 1$. With \eqref{eq:A2} and
\eqref{eqn:depth} we calculate
\[
\begin{aligned}
& a(\mathrm{Ad} \circ \phi)  = \frac{3}{4} ((r + 1) \cdot 4) = 3 (r+1) , \\
& d(\phi) = d(\mathrm{Ad} \circ \phi) = \frac{3 (r+1)}{3} - 1 = r .
\end{aligned}
\]
Like in the octahedral case, the numbers $\Gal (E^+/K)_j$ for $\varphi$ are related
to those for $\phi$ via the twofold covering $\SL_2 (\C) \to \PGL_2 (\C)$.
We find 
\[
g_0 = \cdots = g_r = 8 \text{ and } g_{r+1} = \cdots = g_{2r+1} = 2.
\]
Moreover $g_{2r+2} = 1$ if we assume that the depth of $z$ is $2r+1$.
Now \eqref{eq:A2} says
\[
a(\varphi) = \frac{2}{8} \big( (r+1) \cdot 8 + (r+1) \cdot 2 \big) = 5 (r+1) / 2 \in \Z ,
\]
and from \eqref{eqn:depth} we obtain
\[
d(\varphi) = \frac{5 (r+1)}{2 \cdot 2} - 1 = \frac{5 r + 1}{4} . 
\]
Again, this is larger than $d(\phi) = r$.\\[2mm]

\emph{Imprimitive representations}.  Consider an imprimitive totally ramified representation $\varphi : \W_K \to \GL_2 (\C)$.
By \cite[\S 41.4]{BH} there exists a separable totally ramified quadratic extension
$L/K$ and a character $\xi$ of $\W_L$ such that $\varphi = \mathrm{ind}_{\W_L}^{\W_K}(\xi)$.
Let $\mathfrak p_K^{1 + d(L/K)}$ be the discriminant of $L/K$. If $L \cong
K[X]/(X^2 + X + b)$, then one deduces from \cite[\S 41.1]{BH} that $d(L/K) = -\nu_K (b) > 0$.

From the proof of \cite[Lemma 41.5]{BH} one sees that the level of $\varphi$ equals
$d(\xi) + d(E/F)$. By construction the level of a $n$-dimensional irreducible
representation of $\W_K$ equals $n$ times its depth, so
\begin{equation}\label{eq:ddd}
d(\varphi) = \big( d(\xi) + d(L/K) \big) / 2 .  
\end{equation}
As before we assume that $\varphi$ is minimal among the lifts of $\phi$. Then
\cite[\S 41.4]{BH} says that $d(\xi) > d(L/K)$, and in particular $d(\xi) \geq 2$.
Since $\Gal (K_\sep / L)^2$ is a pro-2-group, the image of $\xi$ in $\C^\times$ is a 
subgroup of even order. 

Let $\sigma$ be the nontrivial element of Gal$(L/K)$, so that the restriction of
$\varphi$ to $\W_L$ is $\xi \oplus \sigma (\xi)$. If $\xi (w) = -1$, then also 
$\xi (\sigma (w)) = -1$. As $\xi (\W_L)$ is even, this means that $\matje{-1}{0}{0}{-1} 
\in \phi (\W_L)$. We note that, as every $\W_K \setminus \W_L$ 
interchanges $\xi$ and $\sigma (\xi)$, the kernel of $\phi$ equals the kernel of
$\xi \oplus \sigma (\xi)$ composed with the projection $\GL_2 (\C) \to \PGL_2 (\C)$.
Thus the kernel of $\phi$ contains the kernel of $\varphi$ with index two. More 
precisely 
\[
\ker (\phi) = (\xi \oplus \sigma (\xi))^{-1} \big\{ \matje{1}{0}{0}{1}, 
\matje{-1}{0}{0}{-1} \big\} = \xi^{-1} \{1,-1\} = \ker (\xi^2) .
\]
By the same argument as above also $\ker (\mathrm{ind}_{\W_L}^{\W_K} \xi^2 ) = \ker (\xi^2)$. 
Hence $\phi$ and $\mathrm{ind}_{\W_L}^{\W_K} (\xi^2 )$ have the same kernel, and in
particular the same depth. With \eqref{eq:ddd} we can express it as
\begin{equation}\label{eq:ddd2}
d(\phi) = \big( d(\xi^2) + d(L/K) \big) / 2. 
\end{equation}
The depth (or level) of $\xi$ is the least $l$ such that $\xi$ (or rather its composition
with the Artin reciprocity isomorphism) is nontrivial on the higher units group
$U_L^l = 1 + \mathfrak p_L^l \subset L^\times$. For $l > 0$ the group $U_L^l / U_L^{l+1}$ 
has exponent 2, so $\xi (U_L^{d(\xi)}) = \{1,-1\}$. Consequently $U_L^{d(\xi)} \subset
\ker \xi^2$ and $d(\xi^2) < d(\xi)$. Comparing \eqref{eq:ddd} and \eqref{eq:ddd2},
we get
\[
d(\varphi) - d(\phi) = \big( d(\xi) - d(\xi^2) \big) / 2 > 0 . \qedhere 
\]
\end{proof}

\section{$L$-packets}   According to a classical result of Shelstad \cite[p.200]{She}, for $F$ of characteristic zero 
all the $L$-packets $\Pi_\varphi (\SL_2(F))$ have cardinality $1, 2$ or $4$. We will check 
below, after \eqref{eq:L1},
that the same holds for the $L$-packets for $\SL_2(K)$. It will follow from the classification 
in this section that $L$-packets for $\SL_1 (D)$ have cardinality 1 or 2.

\begin{thm}\textup{\cite{ABPS1}} \label{thm:3.1}  
Let $\phi : \W_K \times \SL_2 (\C) \to \PGL_2 (\C)$ be an L-parameter for $\SL_2 (K)$, and let 
$\varphi : \W_K \times \SL_2 (\C) \to \GL_2 (\C)$ be a lift of minimal depth. For any $\pi$ in 
one of the $L$-packets $\Pi_{\varphi}(\GL_2 (K))$, $\Pi_\varphi (GL_1 (D))$, $\Pi_\phi (\SL_2 (K))$ and 
$\Pi_\phi (\SL_1 (D))$:
\[
d(\phi) \leq d (\varphi) = d(\pi) .
\]
Moreover $d(\phi) = d (\varphi) = d(\pi)$ if $\varphi$ is essentially tame, in particular whenever
$\varphi$ is unramified.
\end{thm}

We define the groups
\begin{equation}\label{eq: S group}
\begin{aligned}
& C(\phi) := Z_{\SL_2(\C)}(\text{im } \phi) , \\
& \mc S_{\phi} := C(\phi) / C(\phi)^\circ = \pi_0 (Z_{\SL_2 (\C)}(\phi)) , \\
& \mc Z_{\phi} := \rZ(\SL_2(\C)) / \rZ(\SL_2(\C)) \cap C(\phi)^\circ , \\
& S_\phi := \pi_0 (Z_{\PGL_2 (\C)}(\phi)) .
\end{aligned}
\end{equation}
The group $S_\phi$ is abelian, $\mc S_\phi$ can be nonabelian, and there is a short exact sequence 
\begin{equation}\label{eq:L2}
1 \to \mc Z_\phi \to \pi_0 (Z_{\SL_2 (\C)}(\phi)) \to \pi_0 (Z_{\PGL_2 (\C)}(\phi)) \to 1 .
\end{equation}
It is easily seen that $|\mc Z_\phi| = 2$ if and only if $\phi$ is relevant for $\SL_1 (D)$.
By \cite[Theorem 3.3]{ABPS2} there are bijections
\begin{equation}\label{eq:L1}
\begin{aligned}
& \Irr \big( \pi_0 (Z_{\PGL_2 (\C)}(\phi)) \big) \longleftrightarrow \Pi_\phi (\SL_2 (K)) ,\\
& \Irr \big( \pi_0 (Z_{\SL_2 (\C)}(\phi)) \big) \longleftrightarrow 
\Pi_\phi (\SL_2 (K)) \cup \Pi_\phi (\SL_1 (D)) .\\
\end{aligned}
\end{equation}
We remark for $\SL_2 (F)$ with char$(F) = 0$, \eqref{eq:L1} was shown in \cite[Theorem 4.2]{GeKn}
and \cite[Theorem 12.7]{HiSa}.
Recall that $\mf{T}(\varphi)$ is the abelian group of characters $\chi$ of $\W_K$ with 
$\varphi \otimes \chi \cong \varphi$. By \cite[Theorem 4.3]{GeKn} and by \cite[(21)]{ABPS2} 
\begin{equation}\label{eqn:size}
\mf T (\varphi) \cong \pi_0 (Z_{\PGL_2 (\C)}(\phi)). 
\end{equation}
By \cite[Proposition 41.3]{BH}, and by the classification of L-parameters for the principal 
series in Subsection \ref{par:Lone}, $\mf T (\varphi)$ has order dividing four.
This shows that all L-packets for $\SL_2 (K)$ have order 1, 2 or 4.

\subsection{Stability} \

Before we proceed with the classification of $L$-packets, some remarks about the stability of
the associated distributions are in order. In this subsection $K$ can be any local non-archimedean field.
Recall that a class function on an algebraic $K$-group $\cG (K)$ is called stable if it is constant on 
the intersection of any $\cG (K_s)$-conjugacy class with $\cG (K)$. For an invariant distribution
on $\cG (K)$ one would like to use a similar definition of stability, but that does not work well in
general. Instead, stable distributions are usually defined in terms of stable orbital integrals.
But, whenever an invariant distribution $\delta$ on $\cG (K)$ is represented by a class function on an open
dense subset of $\cG (K)$, we can use the easier criterion for stability of functions to 
determine whether or not $\delta$ is stable.

Harish-Chandra proved that the trace of an admissible representation is a distribution which is 
represented by a locally constant function on the set of regular semisimple elements of $\cG (K)$,
see \cite{DBHCS}.
So the study the stability of traces of $\cG (K)$-representations, it suffices to look at
(regular) semisimple elements of $\cG (K)$.

For semisimple elements in $\GL_2 (K)$ conjugacy is the same as stable conjugacy, it is determined by 
characteristic polynomials. Hence every irreducible (admissible) representation of $\GL_2 (K)$ defines 
a stable distribution.

The semisimple conjugacy classes in $\GL_1 (D)$ are naturally in bijection with the elliptic conjugacy 
classes in $\GL_2 (K)$, i.e. those semisimple classes for which the characeristic polynomials are
irreducible over $K$. Moreover any irreducible essentially square-integrable representation of 
$\GL_2 (K)$ is already determined by the values of its trace on elliptic elements. These observations 
constitute some of the foundations of the Jacquet--Langlands correspondence \cite{JaLa}. In fact the 
Jacquet--Langlands correspondence can be defined as the unique bijection between $\Irr (\GL_1 (D))$ and 
the essentially square-integrable representations in $\Irr (\GL_2 (K))$ which preserves the traces on 
elliptic conjugacy classes, up to a sign. Consequently the trace of any irreducible representation $\pi$
of $\GL_1 (D)$ is the restriction of a stable distribution on $\GL_2 (K)$ to the set of elliptic
elements. In particular the trace of $\pi$ is itself a stable distribution.

\begin{prop}\label{prop:3.2}
Let $\phi$ be a L-parameter for $\SL_2 (K)$.
\begin{itemize}
\item[(a)] Write $\Pi_\phi (\SL_2 (K)) = \{ \pi_1, \ldots, \pi_m\}$.
The trace of $\pi := \pi_1 \oplus \cdots \oplus \pi_m$ is a stable distribution on $\SL_2 (K)$.
Any other stable distribution that can be obtained from $\Pi_\phi (\SL_2 (K))$ is a scalar 
multiple of the trace of $\pi$.
\item[(b)] Suppose that $\phi$ is relevant for $\SL_1 (D)$ and write $\Pi_\phi (\SL_1 (D)) =
\{\pi'_1, \ldots, \pi'_{m'} \}$. The trace of $\pi' := \pi'_1 \oplus \cdots \oplus \pi'_{m'}$ is a stable 
distribution on $\SL_1 (D)$. Any other stable distribution that can be obtained from $\Pi_\phi (\SL_1 (D))$ 
is a scalar multiple of the trace of $\pi'$. 
\end{itemize}
\end{prop}
\begin{proof}
(a) Since the restriction of irreducible representations from $\GL_2 (K)$ to $\SL_2 (K)$ is multiplicity-free 
\cite[\S 1]{BuKu}, $\pi = \pi_1 \oplus \cdots \oplus \pi_m$ is the restriction of some irreducible 
representation of $\GL_2 (K)$. If $\varphi : \W_K \times \SL_2 (\C) \to \GL_2 (\C)$ is any lift of
$\phi$, the image of $\phi$ under the local Langlands correspondence is such a representation. We
denote this representation of $\GL_2 (K)$ again by $\pi$. By the above remarks, its trace is a stable
distribution on $\GL_2 (K)$, and hence also on $\SL_2 (K)$. 

The different $\pi_i$ are inequivalent, but they are $\GL_2 (K)$ conjugate, because $\pi$ is irreducible.
If a linear combination $\sum_{i=1}^m \lambda_i \mathrm{tr}(\pi_i)$ is a stable distribution, 
then it must be invariant under conjugation by $\GL_2 (K)$. Hence all the $\lambda_i \in \C$ must be equal.

(b) The restriction of representations from $\GL_1 (D)$ to $\SL_1 (D)$ can have multiplicities,
but still every constituent will appear with the same multiplicity \cite[Lemma 2.1.d]{GeKn}. So
there exists an integer $\mu$ such that $\mu \pi' = \mu \pi'_1 \oplus \cdots \oplus \mu \pi'_{m'}$ lifts
to an irreducible representation of $\GL_1 (D)$. The L-parameter of such a representation is a lift
of $\phi$, so we can take JL$(\pi)$, the image of $\pi$ under the Jacquet--Langlands correspondence. 

As remarked above, tr(JL$(\pi))$ is stable distribution on $\GL_1 (D)$ and by restriction also on 
$\SL_1 (D)$. Thus tr$(\pi') = \mu^{-1} \text{tr(JL)} (\pi))$ is also a stable distribution on $\SL_1 (D)$. 
By the same argument as for part (a), any linear combination of the tr $(\pi'_i)$ which is stable, 
must be a scalar multiple of tr$(\pi')$. 
\end{proof}

We remark that Proposition \ref{prop:3.2} also holds for inner forms of $\SL_n (F)$ with $n>2$.
The proof is the same, one only has to replace the elliptic conjugacy classes by the conjugacy
classes that correspond to elements of that particular inner form.

\subsection{$L$-packets of cardinality one} \label{par:Lone} \

First we consider the case that $\varphi : \W_K \to \GL_2 (\C)$ is irreducible, so the
$L$-packet consists of supercuspidal representations.
By \eqref{eqn:size} and \eqref{eq:L1}, $\Pi_\phi (\SL_2 (K))$ is a singleton if
and only if $\varphi$ is primitive. The L-parameter $\phi$ is relevant for $\SL_1 (D)$, so 
$\Pi_\phi (\SL_1 (D))$ is nonempty. It follows from \eqref{eq:L1} and \eqref{eq:L2} that 
$\mc Z_\phi \cong \pi_0 (Z_{\SL_2 (\C)}(\phi)) \cong \Z / 2 \Z$, and then from \eqref{eq:L1}
that $\Pi_\phi (\SL_1 (D))$ is also a singleton. Any primitive representation of $\W_K$ is
either octahedral or tetrahedral, as in Section \ref{sec:depth}. See \cite[\S 42]{BH}
for more background.

Suppose now that $\varphi : \W_K \to \GL_2 (\C)$ is reducible, so $\phi$ is a L-parameter
for the principal series of $\SL_2 (K)$. If $\phi (\W_K) = 1$ and $\phi |_{\SL_2 (\C)} :
\SL_2 (\C) \to \PGL_2 (\C)$ is the canonical projection, then $\phi$ is relevant for
$\SL_1 (D)$. In this case $\Pi_\phi (\SL_1 (D))$ is just the trivial representation of
$\SL_1 (D)$, and $\Pi_\phi (\SL_2 (K))$ consists of the Steinberg representation of
$\SL_2 (K)$ -- the unique irreducible square-integrable, non-supercuspidal representation.

All other principal series L-parameters are trivial on $\SL_2 (\C))$ and are irrelevant
for $\SL_1 (D)$. By conjugating $\phi$, we may assume that its image is contained in the
diagonal torus of $\PGL_2 (\C)$. One checks that $Z_{\PGL_2 (\C)}(\phi)$ is connected unless 
the image of $\phi$ is $\{ 1, \matje{-1}{0}{0}{1} \}$. Whenever $Z_{\PGL_2 (\C)}(\phi)$ 
is disconnected, its L-packet has two elements, see Subsection \ref{par:princ}.  

If $Z_{\PGL_2 (\C)}(\phi)$ is connected, then $\Pi_\phi (\SL_2 (K))$ consists of precisely 
one principal series representation. Let $T$ be the diagonal torus of $\SL_2 (K)$, and let 
$\chi_\phi$ be the character of $T$ determined by local class field theory. Then
$\Pi_\phi (\SL_2 (K))$ is the Langlands quotient of the parabolic induction of $\chi_\phi$,
and the depth of that representation equals the depth of $\chi_\phi$.

\subsection{Supercuspidal $L$-packets of cardinality two} \
 
For such L-parameters (\ref{eqn:size}) shows that
\[
\mathfrak{T} (\varphi) \cong \pi_0 (Z_{\PGL_2 (\C)}(\phi)) \cong \Z / 2 \Z .
\] 
The L-parameter $\phi$ is relevant for $\SL_1 (D)$, so by \eqref{eq:L2} 
$|\pi_0 (Z_{\SL_2 (\C)}(\phi))| = 4$. Then $\pi_0 (Z_{\SL_2 (\C)}(\phi))$ is either $\Z / 4\Z$
or $(\Z / 2 \Z )^2$. In any case, it is abelian
and has precisely four inequivalent characters. Now \eqref{eq:L1} says that
\[ 
|\Pi_\phi (\SL_1 (D))| = |\Pi_\phi (\SL_2 (K))| = 2.
\]

Now we classify the discrete L-parameters $\phi$ for which the packet $\Pi_\phi (\SL_2 (K))$
is not a singleton.
We note that every L-parameter for a supercuspidal representation of $\SL_2 (K)$ has to
be trivial on $\SL_2 (\C)$. For if it were nontrivial on $\SL_2 (\C)$, then the image of $\W_K$
would be in the centre of $\PGL_2 (\C)$, and we would get the L-parameter for the Steinberg
representation, as discussed in the previous subsection. Since we want $\phi$ to be discrete,
it has to be an irreducible projective two-dimensional representation of $\W_K$.

Let $\varphi$ be an irreducible two-dimensional representation of $\W_K$ which lifts $\phi$. 
By \eqref{eqn:size} and \eqref{eq:L1} the associated $L$-packet $\Pi_\phi (\SL_2 (K))$ has 
more than one element if and only if $\varphi$ is imprimitive. By \cite[\S 41.3]{BH} 
$\varphi$ is imprimitive if and only if there exists a separable quadratic 
extension $E/K$ and a character $\xi$ of $E^\times$ such that $\varphi \cong \Ind_{E/K} \xi$.  
By the irreducibility $\xi^\sigma \neq \xi$, where
$\sigma$ is the nontrivial automorphism of $E$ over $K$.

\begin{lem}\label{lem:3.3}
Let $\phi$ and $\varphi \cong \Ind_{E/K} \xi$ be as above. 
\begin{itemize}
\item[(a)] Suppose that the character $\xi^\sigma \xi^{-1}$ of $E^\times$ has order two.
Then $\varphi$ is triply imprimitive and there exists a biquadratic extension $L / K$
such that $\ker (\phi) = \W_L$ and $L \supset E$.
\item[(b)] Suppose that $\xi^\sigma \xi^{-1}$ has order $>2$. 
Then $\varphi$ is simply imprimitive.
\end{itemize}
\end{lem}
\begin{proof}
Let $\chi_E$ be the unique character of $\W_K$ with kernel $\W_E$. Then 
$\chi_E \in \mathfrak{T}(\varphi)$, this holds in general for 
induction of irreducible representations from subgroups of index two. In particular
$|\mathfrak{T}(\varphi)| \in \{2,4\}$. From \cite[Corollary 41.3]{BH} we see that 
$\mathfrak{T}(\varphi) = \{ 1, \chi_E \}$ if and only if the character $\xi^\sigma \xi^{-1}$ of 
$\W_E$ cannot be lifted to a character of $\W_F$. Since the target group $\C^\times$ is divisible,
this happens if and only if $\xi^\sigma \xi^{-1}$ does not equal 
\[
(\xi^\sigma \xi^{-1})^\sigma = \xi \xi^{-\sigma} = (\xi^\sigma \xi^{-1} )^{-1}.
\]
We conclude that the representation $\varphi = \Ind_{E/K} \xi$ is triply imprimitive
if $\xi^\sigma \xi^{-1}$ has order two and is simply imprimitive otherwise.

Now we focus on the triply imprimitive case. By local class field theory there exists a unique
separable quadratic extension $L/E$ such that $\xi^\sigma \xi^{-1}$ is the associated character
$\chi_L $ of $E^\times$. We consider it also as a character of $\W_E$. Then
\[
\W_L = \ker (\chi_L) = \{ w \in \W_K : \varphi (w) \in Z(\GL_2 (\C)) \} .
\]
Hence $\W_L = \ker (\phi)$ is a normal subgroup of $\W_K$, which means that $L /K$ is a Galois
extension. The explicit form of $\varphi$ entails that the image of $\phi$ is the Klein four group.
Consequently
\begin{equation}\label{biquad}
\mathrm{Gal}(L/K) \cong \W_K / \W_L \cong \phi (\W_K) \cong (\Z / 2 \Z)^2,
\end{equation}
which says that $L/K$ is biquadratic. 
\end{proof}

We remark that the depth of $\varphi = \Ind_{E/K} \xi$ can be computed in the same way as for 
the imprimitive representations in Section \ref{sec:depth}, see in particular \eqref{eq:ddd}.




 

\subsection{Supercuspidal $L$-packets of cardinality four} 
\label{par:four}  We continue with the case when $\varphi$ is triply imprimitive, as in (\ref{biquad}). This means that we have a biquadratic extension $L/K$ and
the Langlands parameter
\begin{equation}\label{klein}
\phi : W_K \to \mathrm{Gal}(L/K) \cong (\Z / 2 \Z)^2 \subset \PGL_2(\C).
\end{equation}
We also have
\[
Z_{\PGL_2 (\C)}(\im \, \phi) = \pi_0 (Z_{\PGL_2 (\C)}(\im \,\phi)) = S_\phi \cong (\Z / 2\Z)^2 . 
\]
This implies, by (\ref{eq:L1}),  that $\Pi_{\phi}(\SL_2(K))$ is a supercuspidal packet of cardinality $4$.



We note the isomorphism $\PGL_2(\C) = \PSL_2(\C)$, and the morphism
\[
\SL_2(\C) \to \PSL_2(\C).
\]
As in \cite[\S 14]{We}, the pull-back $\mathcal{S}_{\phi}$  of $S_{\phi}$ is isomorphic to the 
the group of unit quaternions $\{\pm 1, \pm \i, \pm \j, \pm \k\}$.   This group 
admits four characters  
and one irreducible representation of degree $2$. Only the two-dimensional representation
$\rho_0$ has nontrivial central character.

The parameter $\phi$ creates a  packet with five elements, which are allocated to $\SL_2(K)$ or $\SL_1(D)$ 
according to central characters.   So $\phi$ gives rise to an $L$-packet $\Pi_{\phi} (\SL_2(K))$ with $4$ elements, 
and a singleton packet to the inner form $\SL_1(D)$.

\begin{thm}\label{ddd}  
Let $L/K$ be a biquadratic extension, let $\phi$ be the Langlands parameter (\ref{klein}).
If  $t$ is the highest break in the upper ramification of $\Gal(L/K)$ then $d(\phi) = t$.  
The allowed values of $d(\phi)$  are $1,3,5,7, \ldots$ except in Case 2.2 (see Appendix \ref{par:ram}), 
when the allowed values are $3,5,7, \ldots$.
\end{thm}
\begin{proof}  
From the inclusion $L\subset K_s$ we obtain a natural surjection
\[
\pi_{L/K}\colon\Gal(K_s /K)\rightarrow \Gal(L/K).
\]
Let $K_{\ur}$ be the maximal unramified extension of $K$ in $K_s$ and let 
$K_{\ab}$ be the maximal abelian extension of $K$ in $K_s$.
We have a commutative diagram, where the horizontal maps are the canonical 
maps and the vertical maps are the natural projections
\[
\xymatrix{
1 \ar[r]& I_{K_s /K} \ar[d]_{\alpha_1}\ar[r]^{\iota_1}
& \Gal(K_s /K) \ar[d]_{\pi_1}\ar[r]^{p_1} & \Gal(K_{\ur}/K) \ar[d]_{id}\ar[r]&  1\\
1 \ar[r]& I_{K_{\ab}/K} \ar[d]_{\alpha_2}\ar[r]^{\iota_2}
& \Gal(K_{\ab}/K) \ar[d]_{\pi_2}\ar[r]^{p_2} & \Gal(K_{\ur}/K) \ar[d]_{\beta}\ar[r]& 1\\
1 \ar[r]& \I_{L/K} \ar[r]^{\iota_3}
& \Gal(L/K) \ar[r]^{p_3} & \Gal(L\cap K_{\ur}/K) \ar[r] & 1
}
\]
In the above notation, we have $\pi_{L/K}=\pi_2\circ\pi_1$.
Let
\begin{equation}
\cdots \subset\I^{(2)}\subset \I^{(1)} \subset \I^{(0)}\subset G=\Gal(L/K)
\end{equation}
be the filtration of the
relative inertia subgroup $\I^{(0)}=\I_{L/K}$ of $\Gal(L/K)$, $\I^{(1)}$ is the wild inertia subgroup, 
and so on. Note that $\I^{(r)}$ is the restriction of the filtration $G^r$ of 
$G=\Gal(L/K)$ to the subgroup $\I_{L/K}$, i.e, $\I^{(r)}=\iota_3(G^r)$.
Let
\begin{equation}
\cdots\subset I^{(2)}\subset I^{(1)} \subset I^{(0)}\subset G=\Gal(\overline{K}/K)
\end{equation}
be the filtration of the absolute inertia subgroup $I^{(0)}=I_{K_s /K}$ of 
$\Gal(K^s/K)$, $I^{(1)}$ is the wild inertia subgroup, and so on.

We have
\begin{equation}\label{inertia}
(\forall r ) \; \pi_{L/K} I^{(r)}  = \I^{(r)}
\end{equation}
This follows immediately from the above diagram. Here, we identify $I^{(r)}$ with 
$\iota_1(I^{(r)})$ and $\I^{(r)}$ with $\iota_3(\I^{(r)})$.
(Note that $\alpha$ is \emph{injective}.  Therefore, by (\ref{inertia}), we have
\[
\phi(I^{(r)}) = 1 \iff (\alpha \circ \pi_{L/K})(I^{(r)}) = 1 \iff \alpha(\I^{(r)}) = 1 \iff \I^{(r)} = 1.
\]
The \emph{highest break} $t$ has the property that $I^{(t+1)}  = 1$ and $I^{(t)} \neq I^{(t+1)}$.
It follows that $d(\phi) = t$.

\medskip

\textbf{Case $1$:} There are two ramification breaks occurring at $-1$ and some odd integer $t>0$:
\[
\{1\}=\cdots= \I^{(t+1)}\subset \I^{(t)}=\cdots=\I^{(0)}= \I_{L/K}\subset \Gal(L/K), \quad d(\phi) = t.
\]
The allowed depths are $1$, $3$, $5$, $7$, $\ldots$.

\medskip

\textbf{Case $2.1$:} One single ramification break occurs at some odd integer $t>0$:
\[
\{1\}=\cdots= \I^{(t+1)}\subset \I^{(t)}=\cdots= \I^{(0)}= \I_{L/K}= \Gal(L/K); \quad d(\phi)  = t.
\]
The allowed depths are $1,3,5,7,\ldots$.

\medskip

\textbf{Case $2.2$:} There are two ramification breaks occurring at some odd integers $t_1<t_2$
(with $\I^{(0)}= \I_{L/K}$) :
\[
\{1\}=\cdots= \I^{(t_2+1)}\subset \I^{(t_2)}=\cdots= \I^{(t_1+1)}\subset \I^{(t_1)}=\cdots= \I^{(0)}= \I_{L/K}= \Gal(L/K);
\]
\[
\quad d(\phi) = t_2.
\]
The allowed depths are $3,5,7,9,\ldots$.
\end{proof}
 
Theorem \ref{ddd} contrasts with the case of $\SL_2(\Q_p)$ with $p>2$.  Here there is a unique biquadratic 
extension $L/K$, and the associated L-parameter $\phi : \Gal(L/K) \to \SO_3(\R)$ has depth zero.

\subsection{Principal series $L$-packets of cardinality two} \
\label{par:princ}

Recall from Subsection \ref{par:Lone} that a principal series L-parameter whose $L$-packet
is not a singleton has image $\{ 1, \matje{-1}{0}{0}{1} \}$ in the diagonal torus $T^\vee$ 
of $\PGL_2 (\C)$. Thus it comes from a character $\W_K \to \C^\times$ of order two. Define
\[
\W_K \times \SL_2(\C) \to K^\times
\]
 to be the projection $(g,M) \mapsto g$ followed by the Artin reciprocity map
 \[
 \mathbf{a}_K \colon\W_K \to K^\times.
 \]
 Let $E/K$ be a quadratic extension and let $\chi_E$ be the associated quadratic character of $K^\times$.
Consider the map
\[
K^\times \to \PGL_2(\C), \quad \quad
\alpha \mapsto \left(
\begin{array}{cc}
\chi_E(\alpha) & 0\\
0 & 1
\end{array}
\right)
\]
The composite map
\[
\phi_E\colon \W_K \times \SL_2(\C) \to K^\times \to \PGL_2(\C)
\]
is then an $L$-parameter attached to $\chi_E$.   For the centralizer of the image, we have
\[
Z_{\PGL_2(\C)}(\im \, \phi_E) = N_{\PGL_2 (\C)}(T^\vee) , \quad
S_\phi \cong \cS_\phi = \{1,w\} ,
\]
where $w$ generates the Weyl group of the dual group $\PGL_2(\C)$.   
As there are two characters $1, \epsilon$ of $W = \{1,w\}$, \eqref{eq:L1} says that the 
$L$-packet has cardinality two. There are two
enhanced parameters $(\phi_E, 1)$ and $(\phi_E, \epsilon)$, which parametrize the two elements 
in the $L$-packet $\Pi_{\phi_E} = \Pi_{\phi_E}(\SL_2 (K))$.   We will write
\begin{equation}\label{packet}
\Pi_{\phi_E} = \{\pi^1_E, \pi^2_E\}.
\end{equation}
If $\gamma \in K_s$ is a root of $X^2 - X - \beta \in K[X]$, the quadratic extension $K(\gamma)$ 
is denoted also by $K( \wp^{-1}(\beta) )$, with $\beta \in K$, where $\wp(X) = X^2 - X$.
So the quadratic character
\[
\chi_{n,j} = ( - , u_j\varpi^{-2n-1} +\wp(K) ]
\]
is associated with the quadratic extension  $E = K( \wp^{-1} ( u_j\varpi^{-2n-1} ) )$, see (\ref{ch}) in the Appendix.

Let $E/K$ be a quadratic extension. There are two kinds: the unramified one $E_0=K(\gamma_0)$ 
and countably many totally (and wildly) ramified $E =K(\gamma)$.
The unramified quadratic extension has a single ramification break for $t=-1$.

Let $E/K$ be a quadratic totally ramified extension.   
According to \cite[Proposition 11, p.411 and Proposition 14, p.413]{Da}, there is a single 
ramification break for $t = 2n+1$. Each value $2n+1$  occurs as
a break, with $n \geq 0,1,2,3,\ldots$.   By Theorem \ref{ddd}, adapted to the present case, we have
\[
d(\phi_E) = 2n+1.
\]
Fix a basis $\mathcal{B}=\{u_1, \ldots, u_f\}$ of $\F_q/\F_2$ and let $u_j\in\mathcal{B}$. 
The next result shows how to realise the extension $E/K$.

\begin{thm} 
If $E = K( \wp^{-1} ( u_j\varpi^{-2n-1} ) )$ then
\[
d(\phi_E) = 2n + 1
\]
 with $n = 0,1,2,3,4, \ldots$.
 \end{thm}

\begin{proof}    Let $\mathbf{a}_K : \W_K \to K^\times$ be the Artin reciprocity map.   
Then we have \cite[Theorem 3.6]{ABPS1}:
\[
\mathbf{a}_K(\Gal(K_\sep/K)^l) = U^{\lceil l \rceil}
\]
for all $l \geq 0$, where $\lceil l \rceil$ denotes the least integer greater than or equal to $l$, and $U^i_K$ is the $i$th higher unit group.

We are concerned here with the quadratic character $\chi = \chi_E$ and the associated $L$-parameter $\phi  = \phi_E$.
 The level $\ell(\chi)$ of $\chi$ is the least integer $n \geq 0$ for which $\chi(U_K^{n+1}) = 1$.    Call this integer $N$.    For this integer $N$, we have
\[
N < l \leq N+1 \implies \mathbf{a}_K(\Gal(K_\sep/K)^l) = U_K^{\lceil l \rceil} = U_K^{N+1} \: \textrm{on which} \, \chi \, \textrm{is trivial}
\]
\[
N-1 < l \leq N \implies \mathbf{a}_K(\Gal(K_\sep/K)^l) = U_K^{\lceil l \rceil} = U_K^N \: \textrm{on which} \, \chi \, \textrm{is nontrivial}
\]
The $L$-parameter $\phi$ will factor through $K^\times$ and
we have to consider its depth $d(\phi)$.
Recall: the depth of $\phi$  is the smallest number $d(\phi) \geq 0$ such that $\phi$  is trivial on $\Gal(K_\sep/K)^l$  for all $l > d(\phi)$.  Then $d(\phi) = N$ in view
of the above two implications.   We infer that
\begin{equation}\label{E}
\ell(\chi_E) = d(\phi_E).
\end{equation}

If $\chi$ is the unramified quadratic character given by $\chi(x) = (-1)^{\val_K(x)}$ 
then we will have to allow $N = -1$ in which case $\phi$ has negative depth.

If $E = K( \wp^{-1} ( u_j\varpi^{-2n-1} ) )$ then  $\chi_E = \chi_{n,j}$ and so we have
\begin{equation}\label{EE}
\ell(\chi_E) = \ell(\chi_{n,j}).
\end{equation}
We now compute the level of the quadratic character $\chi_{n,j}$ defined in (\ref{ch}).
Every $\alpha\in U_K^{i}$ has the form $\alpha=1+\varepsilon\varpi^i$, with 
$\varepsilon\in\mathfrak{o}$, and can be expanded in the convergent product
\[
\alpha = \prod\nolimits_{i\geq 1}(1+\theta_i\varpi^i)
\]
for unique $\theta_i\in\F_q$. As we can see in the proof of Theorem \ref{explicit},
$$\dv(1+\theta_{2n+1}\varpi^{2n+1}, u_j\varpi^{-2n-1})=\Tr_{\F_q/\F_2}(u_j\theta_{2n+1})$$
and
$$\dv(1+\theta_i\varpi^i, u_j\varpi^{-2n-1})=0$$
if $i\nmid 2n+1$.
There exists, therefore, an element $\alpha\in U_K^{2n+1}$ such that $\chi_{n,j}(\alpha)\neq 0$ 
and $\chi_{n,j}(U_K^{2n+2})=1$.
We infer that
\begin{equation}\label{EEE}
\ell(\chi_{n,j}) = 2n+1.
\end{equation}
The theorem now follows from (\ref{E}), (\ref{EE}) and (\ref{EEE}).
\end{proof}

We conclude that, if $E = K( \wp^{-1} ( u_j\varpi^{-2n-1} ) )$ , then
\[
d(\pi^i_E) \geq 2n+1
\]
with $i = 1,2$.

It follows that the depths of the irreducible representations $\pi^1_E, \pi^2_E$ 
in the $L$-packet $\Pi_{\phi_E}$ can be arbitrarily large.
For representations  of enormous depth, such as the ones encountered in this article, 
the term \emph{hadopelagic} commends itself, in contrast
to the currently accepted term \emph{epipelagic} for representations of modest depth, 
see \url{en.wikipedia.org/wiki/Epipelagic}.

\appendix

\section{Artin-Schreier symbol}  
Let $K$ be a local field of characteristic $p$ with finite
residue field $k$. The field of constants $k=\F_q$ is a finite extension of $\F_p$, with degree $[k:\F_p]=f$ and $q=p^f$.
Let $\mathfrak{o}$ be the ring of integers in $K$ and $\mathfrak{p}\subset\mathfrak{o}$ the maximal ideal. A choice of uniformizer $\varpi\in\mathfrak{o}$ determines isomorphisms $K\cong\F_q((\varpi))$, $\mathfrak{o}\cong\F_q[[\varpi]]$ and $\mathfrak{p}=\varpi\mathfrak{o}\cong\varpi\F_q[[\varpi]]$.
The group of units is denoted by $\mathfrak{o}^{\times}$ and $\nu$ represents a normalized valuation on $K$, so that $\nu(\varpi)=1$ and $\nu(K)=\Z$.

 Following \cite[IV.4 - IV.5]{FV}, we have the reciprocity map
 \[
 \Psi_K : K^\times \to \Gal(K_{\ab}/K)
 \]
We define the map (Artin-Schreier symbol)
\[
(-,-] : K^\times \times K \to \F_p
\]
by the formula
\[
(\alpha,\beta] = \Psi_K(\alpha)(\gamma) - \gamma
\]
where $\gamma$ is a root of the polynomial $X^p - X - \beta$.   The polynomial $X^p - X$ is denoted $\wp(X)$.   According to
\cite[p.148]{FV} the pairing $(-,-]$ determines the nondegenerate pairing
\begin{equation}\label{AS}
K^\times / K^{\times p} \times K/\wp(K) \to \F_p.
\end{equation}
Let us fix a coset $\beta + \wp(K) \in K/\wp(K)$.   According to (\ref{AS}), this coset determines an element of $\Hom(K^\times / K^{\times p}, \F_p)$.

Now specialise to $p = 2$.   We will identify the additive group $\F_2$ with the multiplicative group $\mu_2(\C) = \{1, -1\} \subset \C$.  In that case, the elements of
$\Hom(K^\times/K^{\times 2}, \F_2)$ are precisely the quadratic characters of $K^\times$.   Since the pairing (\ref{AS}) is nondegenerate, the quadratic characters are parametrised
by the cosets $\beta + \wp(K) \in K/\wp(K)$.   Now the index of $\wp(K)$ in $K$ is infinite; in fact, the powers $\{\varpi^{-2n -1} : n \geq 0\}$
are distinct coset representatives, see \cite[p.146]{FV}. 

\begin{lem}\label{quadratic} 
For $K=\F_2((\varpi))$ the set of powers $\{\varpi^{-2n -1} : n \geq 0\}$ is a complete set of coset representatives. 
\end{lem}

That is not the case when $K=\F_q((\varpi))$ has residue degree $f>1$. 
Let $\mathcal{B}=\{u_1, \ldots, u_f\}$  denote a basis of the $\F_2$-linear space $\F_q$. Then,
$$\{u_j\varpi^{-2n -1} : n \geq 0, j=1, \ldots, f\}$$
is a complete set of coset representatives of $K/\wp(K)$, see $\S 5$ and $\S6$ of \cite{Da}.

The pairing (\ref{AS}) creates a sequence of quadratic characters
\begin{equation}\label{ch}
\chi_{n,j}(\alpha) : = (\alpha, u_j\varpi^{-2n -1} + \wp(K)]
\end{equation}
with $n \geq 0$ and $j=1, \ldots, f$.

\subsection{Explicit formula for the Artin-Schreier symbol} \

In \cite[Corollary 5.5, p.148]{FV}, the authors introduce the map $d_{\varpi}$ which we now describe.
Let $\varpi$ be a fixed uniformizer. Using the isomorphism $K = \F_q((\varpi))$, 
where $q=2^f$, every element $\alpha\in K$ can be uniquely expanded as

\begin{equation}\label{expansion}
\alpha=\sum_{i\geq i_a}\vartheta_i\varpi^i \, , \,\, \vartheta_i \in\F_q.
\end{equation}
Put
\[
\frac{d\alpha}{d\varpi}=\sum_{i\geq i_a}i\vartheta_i\varpi^i \, , \,\, \res_{\varpi}(\alpha)=\vartheta_{-1}.
\]
Define the pairing
\begin{equation}\label{pairing dv}
\dv:K^{\times}\times K\rightarrow\F_2 \, , \,\, \dv(\alpha,\beta)=\Tr_{\F_q/\F_2}\res_{\varpi}(\beta\alpha^{-1}\frac{d\alpha}{d\varpi})
\end{equation}
By \cite[Theorem 5.6. p.149]{FV}, the pairing $( - , - ]$ coincides with the pairing defined in (\ref{pairing dv}). In particular, $\dv$ does not depend on the choice of uniformizer.

We conclude that every quadratic character $\chi_{n,j}$ from (\ref{ch}) is completely described by
\begin{equation}\label{res 1}
\dv(\alpha,u_j\varpi^{-2n-1})=\Tr_{\F_q/\F_2}\res_{\varpi}(u_j\varpi^{-2n-1}\alpha^{-1}\frac{d\alpha}{d\varpi}) \, , \,\, n\geq 0.
\end{equation}
We seek a formula more explicit than (\ref{res 1}).

\medskip

By \cite[Proposition 5.10, p. 17]{FV}, for every $\alpha\in K^{\times}$ 
there exist uniquely determined $k\in\Z$ and $\theta_i\in\F_q$ for $i\geq 0$ 
such that $\alpha$ can be expanded in the convergent product
\begin{equation}\label{product expansion}
\alpha=\varpi^k \theta_0 \prod_{i\geq 1}(1+\theta_i\varpi^i)
\end{equation}
We have
\begin{multline*}
\dv(\varpi^k  \theta_0 \prod_{i\geq 1}(1+\theta_i\varpi^i), u_j\varpi^{-2n-1}) =\\
\dv(\theta_0\varpi^k,u_j\varpi^{-2n-1})+\dv(\prod_{i\geq 1}(1+\theta_i\varpi^i), u_j\varpi^{-2n-1})
\end{multline*}
Now, $\dv(\theta_0\varpi^k,u_j\varpi^{-2n-1})$ is easy to compute:
\[
\begin{aligned}
\dv(\theta_0\varpi^k,u_j\varpi^{-2n-1})  
& =\Tr_{\F_q/\F_2}\res_{\varpi}(u_j\varpi^{-2n-1}\theta_0^{-1}\varpi^{-k}\frac{d(\theta_0\varpi^k)}{d\varpi}) \\
& =\Tr_{\F_q/\F_2}\res_{\varpi}(ku_j\varpi^{-2n-2}) \\
& = 0 .
\end{aligned}
\]
On the other hand,
\[
\begin{aligned}
\dv(\prod_{i\geq 1}(1+\theta_i\varpi^i), u_j\varpi^{-2n-1}) 
& =\sum_{i\geq 1}\dv(1+\theta_i\varpi^i,u_j\varpi^{-2n-1}) \\
& = \sum_{i=1}^{2n+1}\dv(1+\theta_i\varpi^i,u_j\varpi^{-2n-1})
\end{aligned}
\]
since $\dv(1+\theta_i\varpi^i, u_j\varpi^{-2n-1})=0$ if $i>2n+1$ (see \cite[p. 150]{FV}, proof of Corollary).
Moreover, by the same proof of Corollary in \cite[p. 150]{FV}, we have
\begin{equation}\label{2n+1}
\begin{aligned}
\dv(1+\theta_{2n+1}\varpi^{2n+1}, u_j\varpi^{-2n-1}) & =\Tr_{\F_q/\F_2}((2n+1)u_j\theta_{2n+1}) \\
& =\Tr_{\F_q/\F_2}(u_j\theta_{2n+1}) .
\end{aligned}
\end{equation}
This last formula is a particular case of a more general formula we are about to prove.

In order to compute $\dv(1+\theta_i\varpi^i, u_j\varpi^{-2n-1})$ for $i=1$, $\ldots$, $2n+1$, we need to find the Laurent series expansion of $(1+\theta_i\varpi^i)^{-1}$. This can be done by expanding the geometric series
\[
(1+\theta_i\varpi^i)^{-1}=\sum_{j\geq 0}(-\theta_i\varpi^i)^j=
1-\theta_i\varpi{i}+\theta_i^2\varpi{2i}-\theta_i^3\varpi{3i}+\cdots
\]
We have
\begin{multline*}
u_j\varpi^{-2n-1}(1+\theta_i\varpi^i)^{-1}\frac{d}{d\varpi}(1+\theta_i\varpi^i)= \\
iu_j\theta_i\varpi^{-2n-1+i-1}(1-\theta_i\varpi^i+\theta_i^2\varpi^{2i}-\theta_i^3\varpi^{3i}+\cdots+(-1)^r\theta_i^r\varpi^{ri}+\cdots)$$
\end{multline*}
The residue will be nonzero if
\[
-2n-1+i-1+ri=-1\Leftrightarrow r=\frac{2n+1}{i}-1
\]
Hence, $\dv(1+\theta_i\varpi^i, u_j\varpi^{-2n-1})=0$ if $i\nmid 2n+1$. In particular, $i$ must be odd.

We have:
\[
\dv(1+\theta_i\varpi^i, u_j\varpi^{-2n-1})=\left\{\begin{array}{rrr} 0 & , &\textrm{ if } i\nmid 2n+1 \\
\Tr_{\F_q/\F_2}(u_j\theta_i^{(2n+1)/i}) & , &\textrm{ if } i|2n+1   \end{array}\right.
\]
In particular, we recover formula (\ref{2n+1}) by taking $i=2n+1$.

\medskip

From the above, we have established the following explicit formula.

\begin{thm}\label{explicit}  Let $K$ be a local function field of characteristic $2$ with residue degree $f$, and let  $\chi_{n,j}$ denote the quadratic character from (\ref{ch}).  Then we have the explicit formula
\[
\chi_{n,j}(\alpha) =\sum_{i|2n+1}\Tr_{\F_q/\F_2}(u_j\theta_i^{(2n+1)/i})
\]
where $\alpha=\varpi^k\theta_0\prod_{i\geq 1}(1+\theta_i\varpi^i)\in K^{\times}$, $n \geq 0$ and $j=1, \ldots, f$.
\end{thm}

For example, we have
\[
\chi_{0,1}(\alpha) = \Tr_{\F_q/\F_2}\, \theta_1, \quad \chi_{1,1}(\alpha) = \Tr_{\F_q/\F_2} (\theta^3_1 + \theta_3), \quad \chi_{2,1}(\alpha) = \Tr_{\F_q/\F_2} (\theta^5_1 + \theta_5),
\]
where $\{1, u_2, \ldots, u_f\}$ is a basis of $\F_q/\F_2$.

\subsection{Ramification} \
\label{par:ram}

Quadratic extensions $L/K$ are obtained by adjoining an $\mathbb{F}_2$-line $D\subset K/\wp(K)$. 
Therefore, $L=K(\wp^{-1}(D))=K(\gamma)$ where $D= \span\{\beta+\wp(K)\}$, 
with $\gamma^2-\gamma=\beta$. In particular, if $\beta_0\in\mathfrak{o}\backslash\mathfrak{p}$ 
such that the image of $\beta_0$ in $\mathfrak{o}/\mathfrak{p}$ has nonzero trace in $\F_2$, 
the $\mathbb{F}_2$-line $V_0= \span\{\beta_0+\wp(K)\}$ contains all the cosets $\beta_i+\wp(K)$ 
where $\beta_i$ is an integer and so $K(\wp^{-1}(\mathfrak{o}))=K(\wp^{-1}(V_0))=K(\gamma_0)$ 
where $\gamma_0^2-\gamma_0=\beta_0$ gives the unramified quadratic extension, see \cite[Proposition $12$, p. 412]{Da}.

Biquadratic extensions are computed the same way, by considering $\mathbb{F}_2$-planes 
$W= \span\{\beta_1+\wp(K), \beta_2+\wp(K)\}\subset K/\wp(K)$. Therefore, 
if $\beta_1+\wp(K)$ and $\beta_2+\wp(K)$ are $\mathbb{F}_2$-linearly independent then 
$K(\wp^{-1}(W)):=K(\gamma_1, \gamma_2)$ is biquadratic, where $\gamma_1^2-\gamma_1=\beta_1$ 
and $\gamma_2^2-\gamma_2=\beta_2$, $\gamma_1,\gamma_2\in K^s$. Therefore, $K(\gamma_1, \gamma_2)/K$ 
is biquadratic if $\beta_2-\beta_1\not\in\wp(K)$.

A biquadratic extension containing the line $V_0$ is of the form $K(\gamma_0,\gamma)/K$. 
There are countably many quadratic extensions $L_0/K$ containing the unramified quadratic 
extension. They have ramification index $e(L_0/K)=2$. And there are countably many biquadratic 
extensions $L/K$ which do not contain the unramified quadratic extension. They have ramification index $e(L/K)=4$.

So, there is a plentiful supply of biquadratic extensions $K(\gamma_1, \gamma_2)/K$.

The space $K/\wp(K)$ comes with a filtration
\begin{equation}\label{Filtration K/P(K)}
0\subset_1 V_0\subset_f V_1=V_2\subset_f V_3=V_4\subset_f \cdots\subset K/\wp(K)
\end{equation}
where $V_0$ is the image of $\mathfrak{o}_K$  and $V_i$ ($i>0$) is the image of $\mathfrak{p}^{-i}$ under the canonical surjection $K\rightarrow K/\wp(K)$. For $K=\mathbb{F}_q((\varpi))$ and $i>0$, each inclusion $V_{2i}\subset_f V_{2i+1}$ is a sub-$\mathbb{F}_2$-space of codimension $f$. The $\F_2$-dimension of $V_n$ is
\begin{equation}\label{F_2 dim.}
\dim_{\F_2}V_n=1+\lceil n/2 \rceil f,
\end{equation}
for every $n\in\mathbb{N}$, where $\lceil x \rceil$ is the smallest integer bigger than $x$.

\bigskip

Let $L/K$ denote a Galois extension with Galois group $G$. For each $i\geq -1$ we define the $i^{th}$-ramification subgroup of $G$ (in the lower numbering) to be:
$$G_i=\{\sigma\in G: \sigma(x)-x \in\mathfrak{p}_L^{i+1}, \forall x\in\mathfrak{o}_L\}.$$
An integer $t$ is a \emph{break} for the filtration $\{G_i\}_{i\geq -1}$ if $G_t\neq G_{t+1}$. The study of ramification groups $\{G_i\}_{i\geq -1}$ is equivalent to the study of breaks of the filtration.

There is another decreasing filtration with upper numbering $\{G^i\}_{i\geq -1}$ and defined by the \emph{Hasse-Herbrand function} $\psi=\psi_{L/K}$:
$$G^u=G_{\psi(u)}.$$
In particular, $G^{-1}=G_{-1}=G$ and $G^0=G_0$, since $\psi(-1)=-1$ and $\psi(0)=0$.

Now, in analogy with the lower notation, a real number $t\geq -1$ is a \emph{break} for the filtration $\{G^i\}_{i\geq -1}$ if
\begin{equation}\label{break upper}
\forall \varepsilon> 0, \, G^t\neq G^{t+\varepsilon}.
\end{equation}
We define
\begin{equation}\label{eqn: Gt+}
G^{t+} := \bigcap_{r > t} G^{r}.
\end{equation}
Then $t$ is a break of the filtration if and only if $G^{t+}\ne G^t$. The set of breaks of the filtration is
countably infinite and need not consist of integers.

If $G$ is abelian, it follows from Hasse-Arf theorem \cite[p.91]{FV} that the breaks are integers and (\ref{break upper}) is equivalent to
\[
G^t\neq G^{t+1}.
\]
Let $G_2=\Gal(K_2/K)$ be the Galois group of the maximal abelian extension of exponent $2$, $K_2=K(\wp^{-1}(K))$. Since $G_2\cong K^{\times}/K^{\times 2}$ , the
nondegenerate pairing  (\ref{AS}) coincides with the pairing $G_2\times K/\wp(K)\rightarrow \Z/2\Z$.

The profinite group $G_2$ comes equipped with a ramification filtration $(G_2^u)_{u\geq -1}$ in the upper numbering, see \cite[p.409]{Da}. For $u\geq 0$, we have an orthogonal relation \cite[Proposition 17, p.415]{Da}
\begin{equation}\label{orthogonal}
(G_2^u)^{\bot}=\overline{\mathfrak{p}^{-\lceil u \rceil+1}}=V_{\lceil u \rceil-1}
\end{equation}
under the pairing $G_2\times K/\wp(K)\rightarrow \Z/2\Z$.

\medskip
Since the upper filtration is more suitable for quotients, we will compute the upper breaks. By using the Hasse-Herbrand function it is then possible to compute the lower breaks in order to obtain the lower ramification filtration.

According to \cite[Proposition 17]{Da}, the positive breaks in the filtration $(G^v)_v$ occur precisely at integers prime to $p$. So, for $ch(K)=2$, the positive breaks will occur at odd integers. The lower numbering breaks are also integers. If $G$ is cyclic of prime order, then there is a unique break for any decreasing filtration $(G^v)_v$ (see \cite{Da}, Proposition $14$). In general, the number of breaks depends on the possible filtration of the Galois group.

Given a plane $W\subset K/\wp(K)$, the filtration (\ref{Filtration K/P(K)}) $(V_i)_i$ on $K/\wp(K)$ induces a filtration $(W_i)_i$ on $W$, where $W_i=W\cap V_i$. There are three possibilities for the filtration breaks on a plane and we will consider each case individually.

\medskip

\textbf{Case 1:} $W$ contains the line $V_0$, i.e. $L_0=K(\wp^{-1}(W))$ contains the unramified quadratic extension $K(\wp^{-1}(V_0))=K(\alpha_0)$ of $K$. The extension has residue degree $f(L_0/K)=2$ and ramification index $e(L_0/K)=2$. In this case, there is an integer $t>0$, necessarily odd, such that the filtration $(W_i)_i$ looks like
$$0\subset_1 W_0=W_{t-1}\subset_1 W_{t}=W.$$

By the orthogonality relation (\ref{orthogonal}), the upper ramification filtration on $G=\Gal(L_0/K)$ looks like
$$\{1\}=\cdots=G^{t+1}\subset_1G^{t}=\cdots=G^0\subset_1G^{-1}=G$$
Therefore, the upper ramification breaks occur at $-1$ and $t$.


\medskip

The number of such $W$ is equal to the number of planes in $V_t$ containing the line $V_0$ but not contained in the subspace $V_{t-1}$. This number can be computed and equals the number of biquadratic extensions of $K$ containing the unramified quadratic extensions and with a pair of upper ramification breaks $(-1,t)$, $t>0$ and odd. Here is an example.

\begin{ex}
The number of biquadratic extensions containing the unramified quadratic extension and with a pair of upper ramification breaks $(-1,1)$ is equal to the number of planes in an $1+f$-dimensional $\F_2$-space, containing the line $V_0$. There are precisely
$$1+2+2^2+\cdots+2^{f-1}=\frac{1-2^f}{1-2}=q-1$$
of such biquadratic extensions.
\end{ex}

\textbf{Case 2.1:} $W$ does not contains the line $V_0$ and the induced filtration on the plane $W$ looks like
$$0=W_{t-1}\subset_2 W_{t}=W$$
for some integer $t$, necessarily odd.

The number of such $W$ is equal to the  number of planes in $V_t$ whose intersection with $V_{t-1}$ is $\{0\}$. Note that, there are no such planes when $f=1$. So, for $K=\F_2((\varpi))$, \textbf{case 2.1} does not occur.

Suppose $f>1$. By the orthogonality relation, the upper ramification ramification filtration on $G=\Gal(L/K)$ looks like
$$\{1\}=\cdots=G^{t+1}\subset_2G^{t}=\cdots=G^{-1}=G$$
Therefore, there is a single upper ramification break occurring at $t>0$ and is necessarily odd.


\medskip

For $f=1$ there is no such biquadratic extension. For $f>1$, the number of these biquadratic extensions  equals the number of planes $W$ contained in an $\F_2$-space of dimension $1+fi$, $t=2i-1$, which are transverse to a given codimension-$f$ $\F_2$-space.

\medskip

\textbf{Case 2.2:} $W$ does not contains the line $V_0$ and the induced filtration on the plane $W$ looks like
$$0=W_{t_1-1}\subset_1 W_{t_1}=W_{t_2-1}\subset_1 W_{t_2}=W$$
for some integers $t_1$ and $t_2$, necessarily odd, with $0<t_1<t_2$.

The orthogonality relation for this case implies that the upper ramification filtration on $G=\Gal(L/K)$ looks like
$$\{1\}=\cdots=G^{t_2+1}\subset_1G^{t_2}=\cdots=G^{t_1+1}\subset_1G^{t_1}=\cdots=G$$
The upper ramification breaks occur at odd integers $t_1$ and $t_2$.

There is only a finite number of such biquadratic extensions, for a given pair of upper breaks $(t_1,t_2)$.





\begin{thebibliography}{99}
\bibitem[ABPS1]{ABPS1} A.-M. Aubert, P. Baum, R.J. Plymen, M. Solleveld, Depth and the local Langlands correspondence, Arbeitstagung Bonn 2013, 
Progress in Math. 319, Birkhauser 2016,  \url{arxiv.org/abs/
1311.1606}

\bibitem[ABPS2]{ABPS2} A.-M. Aubert, P. Baum, R.J. Plymen, M. Solleveld, The local Langlands correspondence for inner forms of $\SL_n$, Res. Math. Sci., to appear,
 \url{arxiv.org/abs/1305.2638}

\bibitem[BuHe1]{BH} C. Bushnell, G. Henniart, {\em The local Langlands conjecture for $\GL(2)$}, 
Grundlehren der math. Wissenschaften, \textbf{335}, Springer, 2006.

\bibitem[BuHe2]{BHet} C. Bushnell, G. Henniart, {\em The essentially tame local Langlands correspondence, 
III: the general case},  Proc. Lond. Math. Soc. (3) {\bf 101} (2010), no. 2, 497--553.

\bibitem[BuKu]{BuKu} C.J. Bushnell, P.C. Kutzko,  
The admissible dual of $\SL(N)$ II,
Proc. London Math. Soc. (3) {\bf 68} (1994), 317--379.
 

\bibitem[Da]{Da}
C.S. Dalawat, Further remarks on local discriminants, J. Ramanujan Math. Soc., (4) 25 (2010) 393--417.

\bibitem[DBHCS]{DBHCS} S. DeBacker, Harish-Chandra, P.J. Sally,
\emph{Admissible invariant distributions on reductive p-adic groups},
University Lecture Series {\bf 16},
American Mathematical Society, Providence RI, 1991.

\bibitem[Del]{Del} P. Deligne, ``Les corps locaux de caract\'eristique p, limites de corps locaux de caract\'eristique 0", pp. 119--157 in: \emph{Repr\'esentations des groupes r\'eductifs sur un corps local},
Travaux en cours, Hermann, 1984.


\bibitem[FeVo]{FV} I. B. Fesenko and S.V. Vostokov, {\em Local fields and their extensions}, 2nd ed., 
AMS Math. Monograph  Translation 121, 2002.

\bibitem[GeKn]{GeKn} S.S. Gelbart, A.W. Knapp,
L-indistinguishability and R groups for the special linear group,
Adv. in Math. {\bf 43} (1982), 101--121.

\bibitem[GrRe]{GrRe}  B.H. Gross, M. Reeder, Arithmetic invariants of
discrete Langlands parameters, Duke Math. J. \textbf{154} (2010), 431--508.


\bibitem[HiSa]{HiSa} K. Hiraga, H. Saito,
On L-packets for inner forms of $\SL_n$, 
Mem. Amer. Math. Soc. {\bf 1013}, Vol. {\bf 215}, 2012.

\bibitem[JaLa]{JaLa} H. Jacquet, R. Langlands,
\emph{Automorphic forms on $\GL(2)$},
Lecture Notes in Mathematics {\bf 114}, Springer-Verlag, Berlin, 1970.


\bibitem[MoPr]{MoPr} A. Moy, G. Prasad, Unrefined minimal $K$-types for $p$-adic groups,
Inv. Math. {\bf 116} (1994), 393--408.

\bibitem[Ne]{Ne}J. Neukirch, Algebraic Number Theory, Springer-Verlag, Berlin, 1999.


\bibitem[Ser]{Ser} J.-P. Serre, {\em Corps locaux}. Hermann, Paris, 1962.

\bibitem[She]{She} D.~Shelstad, Notes on $L$-indistinguishability, Proc. Symp. 
Pure Math. {\bf 33}, part.2~(1979), 193--203.


\bibitem[We]{We} A. Weil, Exercices dyadiques, Invent. Math. 27 (1974) 1--22.
\end{thebibliography}
\end{document}